\documentclass[10pt,leqno]{amsproc}
\usepackage[utf8]{inputenc}
\usepackage{amsthm}
\usepackage{amsmath}
\usepackage{amssymb}
\usepackage{comment}
\usepackage[margin=1.5in]{geometry}
\usepackage{comment}
\usepackage[dvipsnames]{xcolor}
\usepackage{ragged2e}
\usepackage{blindtext}
\usepackage{xcolor}

\usepackage[colorlinks=true, pdfstartview=FitV, linkcolor=blue, citecolor=blue, urlcolor=blue]{hyperref}

\renewcommand\r{\rangle}
\renewcommand\l{\langle}
\newcommand{\bV}{\boldsymbol{V}}

\newcommand\cal{\mathcal}

\newcommand\R{\mathbb{R}}
\newcommand \C{\mathbb{C}}
\newcommand\Q{\mathbb{Q}}
\newcommand\Z{\mathbb{Z}}
\newcommand\N{\mathbb{N}}

\newcommand\G{{\bf\Gamma}}
 \newcommand\T{\mathbb T}

\newcommand\B{\mathcal{B}}
\newcommand{\E} {{\mathcal E}}

\newtheorem{Thm}{Theorem}[section]
\newtheorem{Lemma}[Thm]{Lemma}
\newtheorem{Ex}[Thm]{Example}
\newtheorem{Cor}[Thm]{Corollary}
\newtheorem{Prop}[Thm]{Proposition}
\theoremstyle{remark}

\begin{document}

     \title{Exponential bases on modified domains}
\author{Oleg Asipchuk}
\address{Oleg Asipchuk,  Florida International University,
	Department of Mathematics and Statistics,
	Miami, FL 33199, USA}
\email{aasip001@fiu.edu}

\begin{abstract}
The stability of exponential bases on domains in $\R^n$ has been widely studied, with much of the research focusing on small perturbations of the phase. In this paper, we consider bases on measurable domains of the real line, and their stability under small modifications of the domain is investigated. 
\end{abstract}
\maketitle

\section{Introduction}  
The main purpose of this paper is to explore the stability of exponential bases under small modifications of the domain.

An exponential basis on a domain $D\subset \R^d$ is an unconditional  Sch\"auder basis for $L^2(D)$ in the form of  $\{e^{2\pi i \lambda_n \cdot x}\}_{n\in\Z^d  }$ with  $\lambda_n \in\R^d$.   A canonical example of exponential basis is the Fourier basis $ \E=\{e^{2\pi i n   x}\}_{n\in\Z   }$  for $L^2([0,1))$. In this case, the domain is an interval on the real line, and the corresponding set of exponent forms a complete and orthonormal basis for $L^2(0,1)$.

The study of exponential bases on domains in $\mathbb{R}^n$ is a topic of significant interest in harmonic analysis and approximation theory. Unfortunately, the existence of exponential bases on general domains remains largely unexplored. For example, we do not know if exponential bases on triangular and circular domains exist or not. The existence of exponential bases on finite unions of intervals in $\R$  and on multirectangular domains (i.e., unions of rectangles with sides parallel to the coordinate axes) of $\R^n$ is proved in \cite{KozmaNitzan2015,KozmaNitzan2016}. See also \cite{Marzo2006}.  The bases produced in these papers and their frame bounds are not explicit, but see e.g. \cite{DeCarli2019,Fuglede1971,Kadec} and also \cite{asipchuk_drezels_2023} for explicit examples of bases. Further examples of domains with exponential bases can be found in \cite{DebernardiLev2022, GREPSTAD20141, LyubarskiiRashkovskii2000} and \cite{DeCarliKumar2015,Lee2024,LeeGotzWalnut2023} and their references. Additionally, sets with no exponential basis were found recently, see \cite{KNO2023}.

A classical problem in this area involves understanding the stability properties of exponential bases under various perturbations. Most existing research has primarily focused on the stability of these bases concerning small changes in the phase or frequency parameters. For examples, refer to \cite{Avdonin1974, Balan1997, Kadec, SunZhou1999} and the references therein.

A less explored yet equally important aspect is the stability of exponential bases when the underlying domain itself is changed.  Precisely,  if $\B=\{e^{2 \pi i \lambda_n x}\}$ is an exponential basis on a domain $D\subset\R^d$ and $\tilde D\subset\R^d$ is such that $|D\setminus \tilde D|< \epsilon$ is $\B$ still a basis on $\tilde D$? The problem does not have a simple answer even for very simple domains.

\subsection{The problem}
In this article, we study the effect of domain modifications on exponential bases on unions of intervals of the real line. The domain choice is motivated by the fact that any set of finite measures in $\R$ can be approximated by finite unions of intervals with rational endpoints. Any such union of intervals can be mapped into a union of intervals of intervals with integer endpoints, and so it is natural to treat unions of intervals with integer endpoints as a representative model. 

Let $s\in \N$ and $M_j\in \N\cup\{0\}$ for $j=0,...,s-1$. Also, we assume that $0=M_0<a_1<...<a_{s-1}$. We consider the set
\begin{equation}\label{E-I}
    \cal I_s=\bigcup_{j=0}^{s-1} [a_j,a_j+1).
\end{equation}
If $a_j$ (mod $s$) are distinct, then 
$$
\B=\left\{e^{2 \pi i \frac{nx}{s}}\right\}_{n\in\Z}=\bigcup_{j=0}^{s-1}\left\{e^{2 \pi i \left(n+\frac{j}{s}\right)x}\right\}_{n\in\Z}
$$
is an orthogonal exponential basis for $L^2(I)$ with $A=B=s$ as shown in Proposition \ref{P-basisMod}. Next, we consider a set of perturbations $\E=\{\epsilon_j\}_{j=0}^{s-1}\subset \Q$ and new union of intervals
$$
\cal I_{s,\E}=\bigcup_{j=0}^{s-1} [M_j+\epsilon_j,M_j+\epsilon_j+1).
$$
We keep the first interval in the same position, i.e. $\epsilon_0=0$. We show with Example \ref{Ex-Theorem-main}  that for some choice of the $\epsilon_j$, $\B$ is no longer a basis for $L^2(\cal I_{s,\E})$.   However,  the next theorem shows that we can obtain a basis on $I_{s,\E}$ from a small perturbation of the basis $\B$. 

\begin{Thm}\label{T-main}
    Let $s\in \N$ and $\{\epsilon_l\}_{l=1}^{s}\subset \Q$ with a property $\sup_l |\epsilon_l|<\frac12$. Let $N$ be the least common denominator of $\{\epsilon_l\}_{l=1}^{s}$ and $\beta$ be a solution of the equation
\begin{equation}\label{E-beta}
    \frac{\sin{(\pi sN \beta)}}{\sin{(\pi \beta)}}=sN\sin{\frac{1}{sN}}
\end{equation}   
    on the interval $\left(0,\frac{1}{sN}\right)$. If  
    \begin{equation}\label{E-delta}
        \frac{1}{2s^2N^3 m } \leq |\delta|\leq \frac{1}{sN^2 m}-\frac{\beta}{N m},
    \end{equation} 
    where $m:=N(a_{s-1}+\epsilon_{s-1})$ and $a_j$ (mod $s$) are distinct, then 
    $$
    \B_{I_{\E},\delta} = \bigcup_{j=0}^{s-1}\left\{e^{2 \pi i \left(n+\frac{j}{s}+j\delta \right)x} \right\}_{n\in\Z}
    $$ 
   is a basis for $L^2(\cal I_{s,\E})$ with Riesz constants 
$$
\frac{1}{N}\left(1-sN\sin{\frac{1}{sN}}\right) \left(sN-  \frac{\sin{\left(\frac{\pi}{2 N m} \right)}}{\sin{\left(\frac{\pi}{2 sN^2 m} \right)}}\right)  \leq A 
$$
and
$$
   B\leq \frac{1}{N}\left(1+sN\sin{\frac{1}{sN}}\right)  \left(sN+  \frac{\sin{\left(\frac{\pi}{2 N m} \right)}}{\sin{\left(\frac{\pi}{2 sN^2 m} \right)}}\right).
$$
\end{Thm}
The existence of $\beta$ that satisfies \eqref{E-beta} and the existence of $\delta$ that satisfies \eqref{E-delta} are justified by Lemma \ref{L-p_k=msN} and Proposition \ref{P-inequalities} in Section 2.

Next, we consider a basis on the union of the intervals $\cal I_M$ inside the interval $I_N=[0,N)$, the solution to which we provide in the next theorem.  In Proposition \ref{P-basis} we prove that
$$
\B_{M} = \bigcup_{j=1}^{M} \left\{e^{2\pi i \left(n+\frac{j-1}{N}\right)x}\right\}_{n\in\Z}
$$
is a basis for $L^2(\cal I_M)$. Estimating the Riesz constants is often the most challenging aspect of the problem, and we address this in the following theorem  
\begin{Thm}\label{T-main-3}
    Let $N,M\in\N$ such that $2<M\leq\frac{N}{2}$. Also, let $u\in\N$ and 
\begin{equation}\label{E-u}
     u > \begin{cases} 
    \frac{N}{\pi} \cos^{-1}{\left(M \sin{\frac{1}{M}}\right)} \,  & N \text{ is even}\\
   \frac{N}{\pi}\cos^{-1}{\left(M \sin{\frac{1}{M}} \cos{\frac{\pi}{2 N}}\right)}- \frac{1}{2} \, & N \text{ is odd}
    \end{cases}.
\end{equation}
If all distinct $a_k$ and $a_{k'}$ satisfy the following inequality
\begin{equation}\label{E-condition-T3-1}
    \min_{m\in \Z} \left\lvert \frac{a_k-a_{k'}}{N} - m \right\rvert - \min_{m\in \Z} \left\lvert \frac{(a_k-a_{k'})M}{N} - m \right\rvert > \frac{u}{N},
\end{equation}
then $\B_{M}$ is a basis for $L^2(\cal I_M)$ with 
\begin{equation*}
    \begin{cases} 
    M\left(1-\left\lvert \cos{\frac{\pi u}{N}} \right\rvert \right) \leq A \text{ and } B \leq   M\left(1+\left\lvert \cos{\frac{\pi u}{N}} \right\rvert \right) \,  & N \text{ is even}\\
   M\left(1-\left\lvert \frac{\cos{\left(\frac{\pi}{2N}+\frac{\pi u}{N}\right)}}{\cos{\frac{\pi}{2N}}} \right\rvert\right)  \leq A \text{ and } B \leq   M\left(1+\left\lvert \frac{\cos{\left(\frac{\pi}{2N}+\frac{\pi u}{N}\right)}}{\cos{\frac{\pi}{2N}}} \right\rvert \right) \, & N \text{ is odd}
    \end{cases}.
\end{equation*}
\end{Thm}

The problem is considerably more difficult because the set of exponentials needs to be adjusted to satisfy the assumption of Beurling's density theorem \cite{Beurling1989}.  In this paper, we consider only a model case but we plan to pursue the investigation of this problem in a subsequent work. 

\begin{Thm}\label{T-main-2}
    Let $2<N\in\N$, $N-1>m\in\N$, $I_N=[0,N)$, and $\beta$ be a solution of the equation
 $$
    \frac{\sin{(\pi (N-1) \beta)}}{\sin{(\pi \beta)}}=(N-1)\sin{\frac{1}{N-1}}.
    $$
If $\frac{1}{2(N-1)^2}<\delta< \frac{1}{N-1}-\beta$ then 
    $$
\B_{\delta} = \bigcup_{j=0}^{N-2} \left\{e^{2\pi i\left(n+\frac{j}{N-1}-j\delta\right)x}\right\}_{n\in\Z},
$$
is a basis for $L^2(I_N\setminus (m,m+1))$ with 
$$
\left(1-(N-1)\sin{\frac{1}{N-1}} \right) \left(N-1 - \frac{1}{\sin{\frac{\pi}{2(N-1)}}} \right)  \leq A  
$$
and 
$$
 B\leq \left(1+(N-1)\sin{\frac{1}{N-1}} \right) \left(N-1 + \frac{1}{\sin{\frac{\pi}{2(N-1)}}} \right).
$$
\end{Thm}

\medskip

The structure of this paper is as follows: In Section 2, we provide a brief overview of the definitions and tools utilized in our research. Section 3 is dedicated to proving the main results. Finally, Section 4 presents a discussion of open problems and concluding remarks.

\medskip
\noindent
{\it Acknowledgments}. First and foremost, I wish to express my deepest gratitude to my scientific advisor, Professor Laura De Carli, for her unwavering support and guidance throughout my research.

The author would also like to acknowledge the financial support provided by the Florida International University Graduate School Dissertation Fellowship.


\section{Preliminaries}

We have used the excellent textbooks \cite{Heil} and  \cite{Young} for the definitions and some of the results presented in this section.

Let  $H$   be a  separable Hilbert space with inner product $\langle\ ,\ \rangle $  and norm $||\ ||=\sqrt{\l \ , \    \r} $. We will mostly work with $L^2(D),$ where $D\subset\R^d.$ So, the norm will be $||f||_{L^2(D)}^2=\int_D |f(x)|^2 dx.$ The characteristic function on $D$ we denote as $\chi_D.$   
  
A sequence of vectors ${\mathcal V}= \{v_j\}_{j\in\Z }$  in   $H$ is a
{\it Riesz basis} if 
there exist  constants $A, \ B>0$    such that, for any $w\in H$ and  for all finite sequences $ \{a_j\}_{j\in J}\subset\C $, the following inequalities hold. 
\begin{align}\label{e2-RS}
 	A  \sum_{j\in J}   |a_j|^2   \leq  &\Big\Vert \sum_{j\in J}  a_j  v_j \Big\Vert^2  \leq B \sum_{j\in J} |a_j|^2.
\\\label{e2-frame}
 A||w||^2\leq  &\sum_{j=1}^\infty |\l  w, v_j\r |^2\leq B ||w||^2.
 \end{align}
The constants $A$ and $B$ are called {\it frame constants} of the basis. The left inequality in \eqref{e2-RS}  implies that ${\cal V}$ is linearly independent,  and the left inequality in \eqref{e2-frame} implies that  ${\cal V}$ is complete. If condition \eqref{e2-RS} holds we call ${\mathcal V}$ a Riesz sequence. We call ${\mathcal V}$ a frame if the condition \eqref{e2-frame} holds. If the condition \eqref{e2-frame} holds and $A=B$ then we call  ${\mathcal V}$ a tight frame. If $A=B=1,$ then we have a Parseval frame.
The following lemma is well-known, but for the reader's convenience, we will prove it.
\begin{Lemma}\label{lmBconstant}
If a sequence of vectors ${\mathcal V}= \{v_j\}_{j\in\Z }$ is a frame with upper constant $B$, then the right inequality in $\eqref{e2-RS}$ holds, i.e. for all finite sequences $ \{a_j\}_{j\in J}\subset\C $ 
$$
\Big\Vert \sum_{j\in J}  a_j  v_j \Big\Vert^2  \leq B \sum_{j\in J} |a_j|^2
$$
\end{Lemma}
\begin{proof}
${\mathcal V}$ is a frame, so for all finite sequences $ \{a_j\}_{j\in J}\subset\C $ there is $f\in H$ such that $f=\sum_{j\in J}  a_j  v_j.$ So,
\begin{equation*}
    \begin{split}
        ||f||^2&=\l\sum_{j\in J}  a_j  v_j,f\r=\sum_{j\in J}  a_j\l v_j,f\r\\
        &\leq \sqrt{\sum_{j\in J}  |a_j|^2} \sqrt{\sum_{j\in J}  |\l v_j,f\r|^2}\leq \sqrt{\sum_{j\in J}  |a_j|^2} \sqrt{B}||f||.
    \end{split}
\end{equation*}
Therefore, 
$$
\Big\Vert \sum_{j\in J}  a_j  v_j \Big\Vert  \leq \sqrt{\sum_{j\in J}  |a_j|^2} \sqrt{B}.
$$
\end{proof}

The next lemma shows that frames in $L^2(D)$ can be obtained from bases of $L^2( D)$, when $D' \supset D$.
\begin{Lemma}
If  $\B$ is  basis for $L^2(D')$ and  $D\subset D'$, then $\B$ is a frame for $L^2(D)$  with at least the same frame constants. In particular, if $\B$ is an orthogonal basis for $L^2(D')$, then it is a tight frame for $L^2(D)$
\end{Lemma}
\begin{proof}
    Let $\B$ is a basis for $L^2(D'),$ then for all $w\in L^2(D')$
    $$
    A||w||^2\leq  \sum_{j=1}^\infty |\l  w, v_j\r |^2\leq B ||w||^2.
    $$
    Also, for any $f\in L^2(D),$ there is $w\in L^2(D')$ such that $f=w\chi_{D}.$ So, the frame inequality holds for any $f\in L^2(D).$ Therefore, $\B$ is a frame for $L^2(D)$  with at least the same frame constants.
\end{proof}

Let $\vec v \in\R^d$ and $\rho>0$; we  denote with  $d_\rho D =\{ \rho x \ : \:   x\in D\}$ and by $ t_{\vec v} D = \{x+\vec v \ : \  x\in D\}$ the dilation and translation of $D$. Sometimes we will write $\vec v+D$  instead of $t_{\vec v} D$ when there is no risk of confusion. 

The following lemma can easily be proved by applying the definitions \eqref{e2-RS} and \eqref{e2-frame}  with ${\cal V}=\{ e^{2\pi i \l x, \lambda_n\r }\}_{n\in\Z}$  and a change of variables in the integrals.
\begin{Lemma}\label{L-extended}
 Let $\vec v\in\R^d$ and $\rho>0$. 
 The set ${\cal V}=\{ e^{2\pi i \l x, \lambda_n\r }\}_{n\in\Z}$ is a Riesz basis for $L^2(D)$ with constants $A$ and $B$ if and only if  the set  $\{ e^{2\pi i \l x, \, \frac 1\rho \lambda_n \r }\}_{n\in\Z} $ is a Riesz basis for  $L^2(t_{\vec v}(d_\rho D ))$ with constants $A \rho^{ d}$ and $B \rho^{ d}$.
\end{Lemma}

\subsection{Exponential bases and Vandermonde matrices}
One method for demonstrating the existence of an exponential basis and calculating their frame constants consists in associating a set of exponentials on the union of intervals with a matrix and showing how the properties of the matrix can be used to derive properties of the set of exponentials. Specifically, for a given $L\in\N,$ we set $I_L=[0,L).$ For $N\in\N,$ we choose $\{a_1,..., a_N\}=I_L\cap\Z.$ Next, we let
$$
\cal I_L:=\bigcup_{j=1}^N \left[a_j,a_j+1\right).
$$

The next theorem is proved in \cite{DeCarli2019}.
    \begin{Thm}\label{T-matrix}
Let $\{a_1,..., a_N\}=\left[0,L\right)\cap\Z$ and $\{\delta_1, ...,\,\delta_N\}\subset\R$.
	The set 
	$$ \B= \B(\delta_1, ...,\, \delta_N)= \bigcup_{j=1}^N  \{e^{ 2\pi i  ( n+ \delta_j) x}\}_{n\in\Z} $$
	is a Riesz basis for $L^2(\cal I_N)$  if and only if  the matrix
$\G=\{ e^{2\pi i  \delta_j a_k}\}_{1\leq j,k\leq L}$ is nonsingular.  The optimal frame constants of $\B$  are the maximum and minimum singular values of $\G$.
\end{Thm}
The more general version of this theorem is in \cite{AsipchukDeCarliLi2024}. Theorem \ref{T-matrix} implies that estimating the Riesz constants for certain exponential bases on unions of intervals of the real line can be reduced to estimating the singular values of a corresponding matrix. This task can be challenging, especially for large matrices. Nonetheless, there are several methods to address this problem, one of which we discuss in the next subsection.

\subsection{Clusterization of matrices}
Vandermonde matrices appear naturally in connection with problems on exponential bases. 

Recall that a {\it Vandermonde matrix} $\bV$ of size $L\times N$  with elements on the complex unit circle $\T=\R/\Z$ is of the form
$$
\bV:=\bV(L,X):=\Big[ \, e^{2\pi i jx_k}\,  \Big]_{\substack{0\leq j<L\\ 1\leq k\leq N}}, 
$$	
with nodes $X=\{x_k\}_{k=1}^N \subset \R$. It is well known that $\bV$ has rank $N$ if and only if  $X$ consists of $N$ distinct elements of $\T$. See for example \cite{MaconSpitzbart1958,Turner1966}. As shown in \cite{Moitra2015} the condition number of a Vandermonde matrices, i.e. ratio between the greatest and smallest singular values, greatly depends on  the pairwise distance of the nodes  according to the {\it wrap-around metric}  
$$|t-s|_\T:=\min_{n\in \Z}|t-s-n|.$$
We say that two nodes $x_k$ and $x_{k'}$ are equivalent if $|x_k-x_{k'}|_\T=0$.

One method for estimating the singular values of Vandermonde matrices is clustering its columns, i.e. grouping the columns of the matrix into groups (clusters) with enough distances (in the "wrap-around" sense) between them. This method is described for example in \cite{BATENKOV2021,LiLiao2021} for different conditions. Rectangular matrices are more suitable for this method, although some results can also be applied to square matrices.

Let $\{L_k\}_k$ be a set of column vectors, we set
$$
\angle_{\min} (L_k,L_{k'}) = \min_{k\neq k'} \cos^{-1}{\left(\frac{|\langle L_k, L_{k'}\rangle |}{||L_k||||L_{k'}||}\right)}.
$$
The following lemma, Lemma 5.1 from \cite{BATENKOV2021}, shows how the singular values of a clustered matrix can be estimated in terms of the singular values of the clusters.

\begin{Lemma}\label{L-Lemma5.1}
    Let $A\in\C^{N \times N}$, such that $A$ is given in the following block form
    $$
    A=[A_1,...,A_M],
    $$
with $A_{k}\in \C^{N\times s_k}$ and $\sum_{k=1}^M s_k=N$. Let $L_k\subset \C^N$ be the subspace spanned by the columns of the
sub-matrix $A_k$.  Assume that for all $1\leq k,k'\leq M$, $k\neq k'$, and $0\leq \alpha \leq \frac{1}{N}$,
\begin{equation}\label{E-angle_condition}
    \angle_{\min} (L_k,L_{k'}) \geq \frac{\pi}{2} - \alpha.
\end{equation}
Let 
$$
\sigma_1\geq ... \geq \sigma_N 
$$
be the ordered collection of all the singular values of $A$, and let
$$
\tilde \sigma_1\geq ... \geq \tilde \sigma_N 
$$
be the ordered collection of all the singular values of the sub-matrices $\{A_k\}$. Then
$$
\sqrt{1-N\alpha} \tilde \sigma_j \leq \sigma_j \leq \sqrt{1+N\alpha} \tilde \sigma_j. 
$$
\end{Lemma}   

Note that condition \eqref{E-angle_condition} can be  modified as
\begin{equation}\label{E-alternative_to_angle_condition}
    \max_{k\neq k'} \langle L_k,L_{k'} \rangle \leq N \sin{\alpha}
\end{equation}
for $||L_k||^2_2=N$.

Lemma \ref{L-Lemma5.1} is quite general and applicable to any matrix. However, for matrices associated to exponential bases, the term $\langle L_k, L_{k'} \rangle$ involves a sum of exponents that is difficult to control. The structure of Vandermonde matrices allows us to express this sum in a more manageable form, as demonstrated in the following lemma.

\begin{Lemma}\label{L-sum_of exponents}
    Let $\{a_1,...,a_m\}\subset \R$, $\delta \in \R$, and $L_k=[1,e^{2 \pi i a_k \delta},...,e^{2 \pi i a_k (m-1) \delta}]^T$, where $T$ denotes the transpose.  Then for all $k\neq k'$ and $\delta(a_k-a_{k'})\not\in\Z$
        \begin{equation}\label{E-sin_ratio}
        |\langle L_k,L_{k'}\rangle | = \left\lvert \frac{\sin{(\pi m (a_k-a_{k'}) \delta)}}{\sin{(\pi(a_k-a_{k'}) \delta)}} \right\rvert.
    \end{equation}
    \begin{proof}
        For all $k\neq k'$ and $\delta(a_k-a_{k'})\not\in\Z$ we use the formula for the sum of a geometric series to write
        \begin{align*}
            |\langle L_k,L_{k'}\rangle | &= |\sum_{j=0}^{m-1} e^{2\pi i (a_k-a_{k'})\delta j} |= \left\lvert \frac{1-e^{2 \pi i m (a_k-a_{k'})\delta}}{1-e^{2 \pi i (a_k-a_{k'})\delta}} \right\rvert \\
            &= \left\lvert \frac{\sin{(\pi m (a_k-a_{k'}) \delta)}}{\sin{(\pi(a_k-a_{k'}) \delta)}} \right\rvert.
        \end{align*}
    \end{proof}
\end{Lemma}

The ratio \eqref{E-sin_ratio} has several properties that will simplify the process of placing nodes into different clusters. 

For our problem,  we consider matrices $\G=\{ e^{2\pi i  \delta (j-1) a_k}\}_{1\leq j,k\leq L}$. These are  Vandermonde matrices with nodes $\{a_k\}_{1\leq k\leq L}$. We distribute these nodes in the following way: we put in one cluster all nodes for which the pairwise ratio \eqref{E-sin_ratio} is large, and we make sure that the pairwise ratio of any two points from different clusters is small.  In order to do this, we need to carefully estimate the ratio in \eqref{E-sin_ratio}, which we will do in the next subsection. 

\subsection{Three useful lemmas}\label{S-UL}
In this subsection, we state and prove several technical lemmas that are instrumental for proving the main results. Each lemma explores the properties of the ratio between two sine functions with different arguments. 
 
\begin{Lemma}\label{L-p_k=msN}
    For every $M\in \N\setminus\{1\}$, the function $g_M(t):=\frac{\sin{(\pi M t)}}{\sin{(\pi t)}}$ is decreasing, non negative, and $g_{M}(t) < M$ on the interval $\left(0,\frac{1}{M}\right)$. Moreover, there is only one $\beta\in \left(0,\frac{1}{M}\right)$ that solve equation 
    $$
    \frac{\sin{(\pi M \beta)}}{\sin{(\pi \beta)}}=M\sin{\frac{1}{M}}.
    $$
    
\end{Lemma}
\begin{proof}
  We use induction. First, if $M=2$, we have that   
  $$
  g_2(t)=\frac{\sin{(2\pi  t)}}{\sin{(\pi t)}} = \frac{2\sin{(\pi  t)}\cos{(\pi  t)}}{\sin{(\pi t)}} = 2\cos{(\pi  t)}.
  $$
  So, $g_2(t)$ is decreasing, non negative, and $g(t) < 2$ on the interval $\left(0,\frac{1}{2}\right)$. Next, we assume that the conclusions of the lemma hold for $M>2$ and using  the trigonometric identities $\sin(a+b)=\sin a \cos b+\sin b \cos a$, we can write
 \begin{align*}
       g_{M+1}(t)&=\frac{\sin{(\pi (M+1)t)}}{\sin{(\pi t)}}\\&=\frac{\sin{(\pi M t)}\cos{(\pi t)}}{\sin{(\pi t)}}+\cos{(\pi M t)}\\
       &=g_{M}(t) \cos{(\pi t)}+\cos{(\pi M t)}.
 \end{align*}
 By assumption, $g_M(t) < M$, and so $g_{M+1}(t) < M \cos(\pi t)+\cos(\pi Mt) \leq M+1$.

 By assumption, $g_M(t)$ is decreasing, $g_M\left(\frac{1}{M}\right)=0$, and $g_M\left(t\right)\rightarrow M$, when $t \rightarrow 0$, so the equation 
  $$
    \frac{\sin{(\pi M t)}}{\sin{(\pi t)}}=M\sin{\frac{1}{M}}.
    $$
    has only one solution on the interval $\left(0,\frac{1}{M}\right)$. 
\end{proof}

\begin{Lemma}\label{L-p_k=1}
  For any $M\in \N \setminus \{1\}$ and $t\leq\frac{1}{2 M^2}$ we have that 
    \begin{equation}\label{E-p_k=1}
        \left\lvert \frac{\sin{(\pi M t)}}{\sin{(\pi(\frac{1}{M} -t))}} \right\rvert< M \sin{\frac{1}{M}}.
    \end{equation}
\end{Lemma}
\begin{proof}
    By Lemma \ref{L-p_k=msN} for $0<t< \frac{1}{M}$ we have that 
    $$
     \left\lvert \frac{\sin{(\pi M t)}}{\sin{(\pi(\frac{1}{M} -t))}} \right\rvert=\frac{\sin{(\pi M t)}}{\sin{(\pi(\frac{1}{M} -t))}} = \frac{\sin{(\pi M (\frac{1}{M} -t))}}{\sin{(\pi(\frac{1}{M} -t))}}>0
    $$
    and the function $t\rightarrow \frac{\sin{(\pi M t)}}{\sin{(\pi(\frac{1}{M} -t))}}$ is increasing. Using the inequality $\frac{2}{\pi} |x| \leq \sin{|x|}\leq |x|$, we obtain 
    \begin{align*}
        \frac{\sin{(\pi M t)}}{\sin{(\pi(\frac{1}{M} -t))}} \leq \frac{\pi M t}{2(\frac{1}{M} -t)} = \frac{\pi M^2 t}{2(1 -M t)} 
    \end{align*}
    for $0<t< \frac{\pi}{2}-\frac{1}{M}$. Since the inequality
    \begin{equation*}
        \frac{\pi M^2 t}{2(1 -M t)} < M \sin{\frac{1}{M}}
    \end{equation*}
   is satisfied when $0<t< \frac{2 \sin{\frac{1}{M}}}{M(\pi + 2 \sin{\frac{1}{M}})}$. 
       
    We show that 
    \begin{equation}\label{E-inequalityforM}
    \frac{2 \sin{\frac{1}{M}}}{M(\pi + 2 \sin{\frac{1}{M}})} > \frac{1}{2 M^2},
\end{equation}  
whenever $M\geq 2$. For $M=2$, direct calculation shows \eqref{E-inequalityforM}.  For $M>2$, we start with the simple estimation
    $$
    \frac{2 \sin{\frac{1}{M}}}{M(\pi + 2 \sin{\frac{1}{M}})} \geq \frac{2 \sin{\frac{1}{M}}}{M(\pi +  \frac{2}{M})} = \frac{M \sin{\frac{1}{M}}}{M(\frac{M\pi}{2} +  1)}.
    $$
 Next, we show that $v(M):= M \sin {\frac{1}{M}}$ is increasing.  Using the Taylor series of $sin t$ and $cos t$, we can  estimate $v'(M)$ as follows
 \begin{align*}
     v'(M) = \sin {\frac{1}{M}} - \frac{\cos {\frac{1}{M}}}{M} \geq \frac{1}{3M^3}-\frac{1}{30M^5}>0.
 \end{align*}
 for all $M > 2$, so $v(M)$ is increasing. Thus,
 $$
 \frac{M \sin{\frac{1}{M}}}{M(\frac{M\pi}{2} +  1)}= \frac{v(M)}{M(\frac{M\pi}{2} +  1)}\geq  \frac{3 \sin{\frac{1}{3}}}{M(\frac{M\pi}{2} +  1)}\geq \frac{3 \sin{\frac{1}{3}}}{M^2(\frac{\pi}{2} +  \frac{1}{3})}.
 $$
 Finally, because 
 $$
 \frac{3 \sin{\frac{1}{3}}}{\frac{\pi}{2} +  \frac{1}{3}}> \frac{1}{2},
 $$
 we have 
 $$
 \frac{2 \sin{\frac{1}{M}}}{M(\pi + 2 \sin{\frac{1}{M}})} > \frac{1}{2 M^2},
 $$
 for $M\geq 2$. 
    Therefore, by \eqref{E-inequalityforM} the inequality \eqref{E-p_k=1} holds for $0<t\leq \frac{1}{2 M^2}$. 
\end{proof}

\begin{Prop}\label{P-inequalities}
    Let $M\geq 2$ and  $\beta\in \left(0,\frac{1}{M}\right)$ be a solution of the equation 
    $$
    \frac{\sin{(\pi M \beta)}}{\sin{(\pi \beta)}}=M\sin{\frac{1}{M}}.
    $$
    Then $\frac{1}{2M^2}<\frac{1}{M}-\beta$.
\end{Prop}
\begin{proof}
    Let $\beta\in \left(0,\frac{1}{M}\right)$ be a solution of 
    $$
    \frac{\sin{(\pi M \beta)}}{\sin{(\pi \beta)}}=M\sin{\frac{1}{M}}.
    $$
    Then $\frac{1}{M}-\beta$ solves 
$$
\left\lvert \frac{\sin{(\pi M t)}}{\sin{(\pi(\frac{1}{M} -t))}} \right\rvert= M \sin{\frac{1}{M}}.
$$
Using Lemma \ref{L-p_k=msN} and simple change of variables, $t:=\frac{1}{M}-t$, we can show that $\left\lvert\frac{\sin{(\pi M t)}}{\sin{(\pi(\frac{1}{M} -t))}} \right\rvert$ is decreasing on the interval $\left(0,\frac{1}{M}\right)$.

By Lemma \ref{L-p_k=1} we have that for $t=\frac{1}{2M^2}$ 
$$
\left\lvert \frac{\sin{(\pi M t)}}{\sin{(\pi(\frac{1}{M} -t))}} \right\rvert> M \sin{\frac{1}{M}}.
$$
These three facts prove that $\frac{1}{2M^2}<\frac{1}{M}-\beta$.
\end{proof}

The final lemma will be used to prove Theorem \ref{T-main-3}.
\begin{Lemma}\label{L-(l-u)}
    Let $N\in N \setminus \{1\}$. The function 
    $$
    \omega_N(t)=\frac{\sin{\frac{\pi (t-u)}{N}}}{\sin{\frac{\pi t}{N}}}
    $$
    is increasing on the interval $0<u<t<\frac{N}{2}$. 
\end{Lemma}

\begin{proof}
    We rewrite our function using simple trigonometric identity
    $$
     \omega_N(t) = \cos{\frac{\pi u}{N}} - \cot{\frac{\pi t}{N}}\sin{\frac{\pi u}{N}}.
    $$
   Since
     $$
     \omega_N'(t) = \frac{\pi}{N}\csc^2{\frac{\pi t}{N}}\sin{\frac{\pi u}{N}}>0,
    $$
    when $0<u<\frac{N}{2}$, $\omega_N(t)$ is an increasing function on the interval $0<u<t<\frac{N}{2}$.
\end{proof}


\section{Proof of the main results}

In this section, we prove Theorems \ref{T-main}, \ref{T-main-3} and \ref{T-main-2}. The proofs are primarily built upon Theorem \ref{T-matrix} and Lemma \ref{L-Lemma5.1}, and have the following key components:
\begin{enumerate}
    \item Associating a set of exponentials on a given domain with a Vandermonde matrix using Theorem \ref{T-matrix};
    \item Distributing the columns of the  matrix among different clusters using lemmas from Subsection \ref{S-UL};
    \item Estimating the singular values of the matrix using Lemma \ref{L-Lemma5.1}. 
\end{enumerate}

\subsection{Proof of Theorem \ref{T-main}}

We start with a union of intervals
$$
\cal I_{s,\E}=\bigcup_{k=0}^{s-1} [a_k+\epsilon_k,a_k+\epsilon_k+1),
$$
where $\E=\{\epsilon_k\}_{k=0}^{s-1}=\left\{\frac{p_k}{q_k}\right\}_{k=0}^{s-1}\subset \Q$ and the $a_k$ (mod $s$) are distinct. Letting $N$ the least common multiple of  the denominators of $\{q_0,...,q_{s-1}\}$, we extend $I_{\E}$ to obtain a union of intervals with integer endpoints
$$
\cal I_{sN,N\E}=\bigcup_{k=0}^{s-1} [(a_k+\epsilon_k)N,(a_k+\epsilon_k+1)N).
$$
By Lemma \ref{L-extended} $\B=\bigcup_{j=0}^{s-1}\left\{e^{2 \pi i \left(n+j\left(\frac{1}{s} +N\delta\right)\right) x} \right\}$ is a basis for $L^2(I_{\epsilon})$ if and only if $\tilde \B=\bigcup_{j=0}^{sN-1}\left\{ e^{2 \pi i \left(n+j\left[\frac{1}{s N} +\delta\right]\right) x} \right\}$ is a basis for $L^2(I_{sN,N\E})$. To show that $\tilde \B$ is a basis for $L^2(I_{sN,N\E})$ we will use Theorem \ref{T-matrix} and Lemma \ref{L-Lemma5.1}.

Indeed, $\cal I_{sN,N\E}$ has the nodes 
$$
p_k\in \{0,...,N-1,(a_1+\epsilon_1)N,...,(a_1+\epsilon_1)N-1,...,(a_{s-1}+\epsilon_{s-1})N,...,(a_{s-1}+\epsilon_{s-1})N-1\}
$$
for $k=0,...,sN-1$. We "wrap around" these nodes to the interval $[0,1)$ by multiplying by $\frac{1}{s N} +\delta$, i.e $p'_{k}=p_k\left(\frac{1}{s N} +\delta\right)$. 
Next, we group the $p_k$ into clusters consisting of either one or two points. To achieve this, we analyze the proximity of each $p_k$ to other nodes using the wrap-around distance as the measure of closeness. Let's fix $k$. If for all $k\neq k'$ the ratio \eqref{E-sin_ratio} is less than $s N \sin{\frac{1}{s N}}$, then $p_k$ is far enough from other points and we can put it in the cluster of size $1$. Similarly, if we fix distinct $k$ and $k'$ such that the ratio \eqref{E-sin_ratio} greater or equal than $s N \sin{\frac{1}{s N}}$  and if for all $k''$ not equal $k$ and $k'$ the ratio \eqref{E-sin_ratio} less than $s N \sin{\frac{1}{s N}}$, then $p_k$ and $p_{k'}$ are far enough from other points and we can put them in the cluster of size $2$. This approach can be used to form clusters of any size, however, the condition $|\epsilon_l| < \frac12$ guarantees that we have clusters at most length $2$. Let's explain why this is the case.  

Let’s examine the ratio in equation \eqref{E-sin_ratio} more closely. In our context $\delta$ is $\frac{1}{sN} +\delta$, so we have
$$
\left\lvert \frac{\sin{(\pi sN (p_k-p_{k'}) (\frac{1}{sN} +\delta)}}{\sin{(\pi (p_k-p_{k'})(\frac{1}{sN} +\delta))}} \right\rvert.
$$
Note that $p_{k}-p_{k'}\in \Z$. So, for small $\delta$ the closest nodes will be those for which there is $m_{k,k'}\in \Z$ such that $p_k-p_{k'}=m_{k,k'} s N$. In this case 
$$
\left\lvert \frac{\sin{(\pi sN (p_k-p_{k'}) (\frac{1}{sN} +\delta)}}{\sin{(\pi (p_k-p_{k'})(\frac{1}{sN} +\delta))}} \right\rvert = \left\lvert \frac{\sin{(\pi m_{k,k'} s^2 N^2 \delta)}}{\sin{(\pi m_{k,k'} s N \delta)}} \right\rvert > s N \sin{\frac{1}{s N}}
$$
for small enough $\delta$. It means that $p_{k}$ and $p_{k'}$ will be in one cluster if there is $m_{k,k'}\in \Z$ such that $p_k-p_{k'}=m_{k,k'} s N$. Therefore, we need to avoid situations when there is a $k''$ for which we can express $p_k-p_{k''}=m_{k,k''} s N$ and $p_{k'}-p_{k''}=m_{k,k''} s N$, with $m_{k,k'},m_{k',k''}\in\Z$.  

According to the condition of the Theorem \ref{T-main} the $a_l(\text{mod } s)$ are distinct and wrap around in $[0,1)$ perfectly (without intersections) when $\delta = 0$ and all $\epsilon_l=0$. However, when we introduce perturbations to our intervals, the situation changes significantly. Those changes produce nodes that will be close to each other in a wrap-around sense, i.e. $p_k-p_{k'}=m_{k,k'} s N$. Therefore, we need to avoid scenarios where, after perturbation, intervals that were not "neighbors" in the wrap-around sense before the perturbation end up having nodes that are close to each other.

Without loss of generality, we consider $a_l=a_{l'}+2 (\text{mod } s)$. Since $I:=[(a_l+\epsilon_l)N,(a_l+\epsilon_l+1)N)$ and $I':=[(a_{l'}+\epsilon_{l'})N,(a_{l'}+\epsilon_{l'}+1)N)$, the intervals $I$ and $I'$contain the nodes $$\{(a_l+\epsilon_l)N, (a_l+\epsilon_l)N+1,...,(a_l+\epsilon_l+1)N-1\}$$ and $$\{(a_{l'}+\epsilon_{l'})N, (a_{l'}+\epsilon_{l'})N+1,...,(a_{l'}+\epsilon_{l'}+1)N-1\}$$ 
respectively. We need to show that there is no $m\in\Z$ such that 
$$
(a_l+\epsilon_l)N -(a_{l'}+\epsilon_{l'}+1)N+1 = m sN.
$$
It means that 
\begin{align*}
    \frac{a_l-a_{l'}-1}{s}+\frac{\epsilon_l-\epsilon_{l'}}{s}+\frac{1}{sN}&\not\in\Z
\end{align*}
and
\begin{align*}
   \frac{(\epsilon_l-\epsilon_{l'}+1)N+1}{sN}&\not\in\Z.
\end{align*}
So, if $\epsilon>0$ and $|\epsilon_l|\leq \epsilon$ and $|\epsilon_{l'}|\leq \epsilon$, it is enough to find the smallest $\epsilon$ such that $\epsilon_{l'}=\epsilon$, $\epsilon_{l}=-\epsilon$, and 
$$
 \frac{(-2\epsilon+1)N+1}{sN}>0.
$$
This is true for $\epsilon<\frac{N+1}{2N}$ or $\epsilon \leq \frac{1}{2}$. Therefore, if $|\epsilon_l|\leq \frac{1}{2}$, for $l=1,...,s-1$, then all nodes $p_k$ can be sorted in the clusters with the maximal length $2$.

Now, we continue with the main proof. Using Lemma \ref{L-Lemma5.1} we represent our main matrix 
$$\G=\left\{e^{2 \pi i p_k \left(\frac{1}{sN}+\delta\right)(j-1)}\right\}_{1\leq k,j\leq sN}$$ 
as
$$
\G=[\G_1,...,\G_{s'},\G'_1,...,\G'_{s'}],
$$
where $\G_m$ are  $1\times sN$ matrices and $\G_{m'}$ are  $2\times sN$ matrices. All $\G_m$ are column vectors, so they have the singular value equal to the square norm of the vector, i.e. 
$$
\sigma_m^2= ||\G_m||^2=sN.
$$
Note that according to our clusterization all $\G'_{m'}$ are associated to two nodes $p_k$ and $p_{k'}$ with the property $\frac{p_k-p_{k'}}{sN}\in\Z$, so 
$$
\G'_{m'}= \begin{pmatrix}
1 &  1\\
e^{2\pi i p_k (\frac{1}{sN}+\delta)} &  e^{2\pi i p_{k'} (\frac{1}{sN}+\delta)}\\
\vdots & \vdots\\
e^{2\pi i p_k (\frac{1}{sN}+\delta) (sN-1)} &  e^{2\pi i p_{k'} (\frac{1}{sN}+\delta) (sN-1)}
\end{pmatrix}.
$$
So,  the singular values of ${\G'}_{m'}$ are the eigenvalues of the matrix 
$$
{\G'}_{m'}^{*} \G'_{m'} = \begin{pmatrix}
    sN & \sum_{j=0}^{sN-1} e^{2\pi i (p_k-p_{k'})(\frac{1}{sN}+\delta)j }  \\
   \sum_{j=0}^{sN-1} e^{-2\pi i (p_k-p_{k'})(\frac{1}{sN}+\delta)j } & s N 
\end{pmatrix}
$$
which can be easily evaluated: Letting $b=\sum_{j=0}^{sN-1} e^{2\pi i (p_k-p_{k'})(\frac{1}{sN}+\delta)j } =\frac{1-e^{2\pi i (p_k-p_{k'})(\frac{1}{sN}+\delta)sN }}{1-e^{2\pi i (p_k-p_{k'})(\frac{1}{sN}+\delta) }}$, then the eigenvalues satisfy the following equation 
$$
\lambda^2 - 2 s N \lambda + s^2 N^2 - b^2=0.
$$
Then 
\begin{equation}\label{E-sigmaP}
    {\sigma'}_{1,2}^2=\lambda_{1,2}=sN \pm |b| = sN \pm \left\lvert \frac{\sin{(\pi sN (p_k-p_{k'}) (\frac{1}{sN} +\delta)}}{\sin{(\pi (p_k-p_{k'})(\frac{1}{sN} +\delta))}} \right\rvert.
\end{equation}
By Lemma \ref{L-p_k=msN} we know that the second term on the right-hand side of \eqref{E-sigmaP} is an increasing function of $(p_k-p_{k'}) \delta$. Therefore, this term reaches its maximum value when $|p_k-p_{k'}|$ is at as small as  possible value, i.e. $|p_k-p_{k'}|=1$. Thus, the minimal and maximal singular values of the sub-matrices can be estimated as 
$$
 sN- \frac{\sin{(\pi s^2 N^2 |\delta|)}}{\sin{(\pi s N |\delta|})}\leq {\sigma'}_{sN}^2\text{ and }{\sigma'}_{1}^2\leq sN+\frac{\sin{(\pi s^2 N^2 |\delta|)}}{\sin{(\pi s N |\delta|})}.
$$
Note that according to the condition \eqref{E-delta}
$$\frac{1}{2s^2N^3 m } \leq |\delta|\leq \frac{1}{sN^2 m}-\frac{\beta}{N m},$$ where we recall that $\beta$ is the only one solution of the equation
 $$
    \frac{\sin{(\pi sN \beta)}}{\sin{(\pi \beta)}}=sN\sin{\frac{1}{sN}}
    $$
    and $m:=a_{s-1}+\epsilon_{s-1}$.
    So, we can simplify our estimations to 
    $$
 sN- \frac{\sin{\left(\frac{\pi}{2 N m} \right)}}{\sin{\left(\frac{\pi}{2 sN^2 m} \right)}}\leq {\sigma'}_{sN}^2\text{ and }{\sigma'}_{1}^2\leq sN+\frac{\sin{\left(\frac{\pi}{2 N m} \right)}}{\sin{\left(\frac{\pi}{2 sN^2 m} \right)}}.
$$
Taking into account the fact that for all $m_{k,k'}$
$$m_{k,k'} s N +1\leq N(M_{s-1}+\epsilon_{s-1})=N m$$
and Lemma \ref{L-p_k=1} we can show that for  $$\frac{1}{2s^2N^3 m } \leq |\delta|\leq \frac{1}{sN^2 m}-\frac{\beta}{N m}$$ there is $\alpha<\frac{1}{sN}$ such that the maximal angle between sub-matrices is greater or equal than $\frac{\pi}{2} - \alpha$, i.e.
$$
\angle_{\min} (L_k,L_{k'}) \geq \frac{\pi}{2} - \alpha,
$$
where 
\begin{equation}\label{E-alpha}
    \alpha= \frac{1}{sN}\frac{\sin{(\pi s N^2 |\delta|m )}}{\sin{(\pi (\frac{1}{s N}-|\delta|Nm))}}.
\end{equation}
Using Lemma \ref{L-Lemma5.1} we obtain that 
$$
(1-sN\alpha) \left(sN-  \frac{\sin{\left(\frac{\pi}{2 N m} \right)}}{\sin{\left(\frac{\pi}{2 sN^2 m} \right)}}\right)  \leq \sigma^2_{sN}  \text{ and }     \sigma^2_{1}\leq (1+sN\alpha) \left(sN+  \frac{\sin{\left(\frac{\pi}{2 N m} \right)}}{\sin{\left(\frac{\pi}{2 sN^2 m} \right)}}\right). 
$$

Using Lemma \ref{L-p_k=1} and the substitution $t:=\frac{1}{sN} -|\delta|Nm$ we can show that $\alpha$ is a decreasing function of $|\delta|$, using the  bounds for delta  \eqref{E-delta} in \eqref{E-alpha}  we obtain
$$
\alpha\leq \frac{1}{sN}\frac{\sin{(\pi s N^2 (\frac{1}{sN^2 m}-\frac{\beta}{N m}) m )}}{\sin{(\pi (\frac{1}{s N}-(\frac{1}{sN^2 m}-\frac{\beta}{N m}) Nm))}} = \frac{1}{sN}\frac{\sin{(\pi sN \beta)}}{\sin{(\pi \beta)}}=\sin{\frac{1}{sN}}
$$

By Theorem \ref{T-matrix} $\tilde \B$ is a basis for $L^2(I_{sN,N\E})$ and the Riesz constants can be estimated as follows:
$$
\left(1-sN\sin{\frac{1}{sN}}\right) \left(sN-  \frac{\sin{\left(\frac{\pi}{2 N m} \right)}}{\sin{\left(\frac{\pi}{2 sN^2 m} \right)}}\right)  \leq A   
$$
and
$$
   B\leq \left(1-sN\sin{\frac{1}{sN}}\right)  \left(sN-  \frac{\sin{\left(\frac{\pi}{2 N m} \right)}}{\sin{\left(\frac{\pi}{2 sN^2 m} \right)}}\right).
$$
By Lemma \ref{L-extended} $\B$ is a basis for $L^2(I_{s,\E})$ with Riesz constants $A'$ and $B'$, that satisfy 
$$
\frac{1}{N}\left(1-sN\sin{\frac{1}{sN}}\right) \left(sN-  \frac{\sin{\left(\frac{\pi}{2 N m} \right)}}{\sin{\left(\frac{\pi}{2 sN^2 m} \right)}}\right)  \leq A'   
$$
and
$$
   B'\leq \frac{1}{N}\left(1-sN\sin{\frac{1}{sN}}\right)  \left(sN-  \frac{\sin{\left(\frac{\pi}{2 N m} \right)}}{\sin{\left(\frac{\pi}{2 sN^2 m} \right)}}\right).
$$
\begin{FlushRight}
$\square$
\end{FlushRight}
\subsection{Proof of Theorem \ref{T-main-3}}
We start with the interval $I_N=[0,N)$. The set
$$
\B_N = \bigcup_{j=1}^{N} \left\{e^{2\pi i \left(n+\frac{j-1}{N}\right)x}\right\}_{n\in\Z}
$$
is an orthogonal basis for $L^2(I_N)$. From $I_N$ we remove a union of intervals
$$
I_M=\bigcup_{k=1}^M [a_k,a_k+1)
$$
with $a_k$ as in \eqref{E-condition-T3-1}. We show that 
$$
\B_M = \bigcup_{j=1}^{M} \left\{e^{2\pi i \left(n+\frac{j-1}{N}\right)x}\right\}_{n\in\Z}
$$
is a basis for $L^2(I_M)$ and we estimate its Riesz constants using Lemma \ref{L-Lemma5.1}. 

To do so, we associate our set of exponents $\B_M$ on the set of intervals with the matrix 
$$
\G_M=\left\{e^{2\pi i a_k \frac{j-1}{N} }\right\}_{1\leq j,k\leq M}
$$
as in Theorem \ref{T-matrix}. This is a Vandermonde matrix with distinct nodes $\frac{a_k}{N}$. Therefore, it is non-singular and by Theorem \ref{T-matrix}, $\B_M$ is a basis for $L^2(I_M)$. 

The conditions on $a_k$ outlined in the theorem align with those in Lemma \ref{L-Lemma5.1}, and we employ the approach used in the proof of Theorem \ref{T-main}, with only minor adjustments to estimate the frame constants of $\B_M$. 

We evaluate the inner product between two different column vectors $v_k$ and $v_{k'}$ of the matrix $\G$, obtaining  
\begin{align*}
    |\langle v_k, v_{k'}\rangle| = \left|\sum_{j=1}^{M} e^{2\pi i (a_k - a_{k'}) \frac{j-1}{N}}\right| = \left| \frac{1-e^{2 \pi i (a_k - a_{k'}) \frac{M}{N}}}{1-e^{2 \pi i (a_k - a_{k'}) \frac{1}{N}}}\right| = \left| \frac{\sin{\left(\pi (a_k - a_{k'}) \frac{M}{N}\right)}}{\sin{\left(\pi (a_k - a_{k'}) \frac{1}{N}\right)}}  \right|.
\end{align*}

The condition \eqref{E-condition-T3-1} guarantees that for all distinct $a_k$ and $a_{k'}$ and for some positive integer $l$ that satisfy the inequality $0<u<l\leq \frac{N}{2}$, the following inequality holds.
$$
\left| \frac{\sin{\left(\pi (a_k - a_{k'}) \frac{M}{N}\right)}}{\sin{\left(\pi (a_k - a_{k'}) \frac{1}{N}\right)}}   \right| \leq  \frac{\sin{\left(\pi \frac{l-u}{N}\right)}}{\sin{\left(\pi \frac{l}{N}\right)}}.
$$
Next, we use Lemma \ref{L-(l-u)} to show that this ratio is an increasing function of $l$, so 
$$
\frac{\sin{\left(\pi \frac{l-u}{N}\right)}}{\sin{\left(\pi \frac{l}{N}\right)}} \leq \frac{\sin{\left(\pi \frac{\frac{N}{2}-\left\{\frac{N}{2}\right\}-u}{N}\right)}}{\sin{\left(\pi \frac{\frac{N}{2}-\left\{\frac{N}{2}\right\}}{N}\right)}},
$$
where $\left\{\frac{N}{2}\right\}$ is a decimal part of $\frac{N}{2}$. Thus, equation \eqref{E-u} guarantees that the following inequality holds.
$$
|\langle v_k, v_{k'}\rangle| < M \sin{\frac{1}{M}}. 
$$
This implies that all distinct columns of the matrix $\G_M$ are sufficiently well-separated, allowing us to apply Lemma \ref{L-Lemma5.1} to estimate the singular values of the matrix and the frame constants of the basis. Therefore, $\B_M$ is a basis for $L^2(\cal I_M)$ with Riesz's constants
\begin{equation*}
    \begin{cases} 
    M\left(1-\left\lvert \cos{\frac{\pi u}{N}} \right\rvert \right) \leq A \text{ and } B \leq   M\left(1+\left\lvert \cos{\frac{\pi u}{N}} \right\rvert \right) \,  & N \text{ is even}\\
   M\left(1-\left\lvert \frac{\cos{\left(\frac{\pi}{2N}+\frac{\pi u}{N}\right)}}{\cos{\frac{\pi}{2N}}} \right\rvert\right)  \leq A \text{ and } B \leq   M\left(1+\left\lvert \frac{\cos{\left(\frac{\pi}{2N}+\frac{\pi u}{N}\right)}}{\cos{\frac{\pi}{2N}}} \right\rvert \right) \, & N \text{ is odd}
    \end{cases}.
\end{equation*}
\begin{FlushRight}
$\square$
\end{FlushRight}
 
\medskip

Theorem \ref{T-main-3} addresses the case where the nodes are distributed across $M$ clusters, with each cluster containing exactly one node. Nevertheless, we can extend the approach outlined in the proof of Theorem \ref{T-main} to handle the case of clusters containing at most two nodes each.

\begin{Cor}
Let $N,M\in\N$ such that $2<M\leq\frac{N}{2}$. Also, let $u\in\N$ satisfies \eqref{E-u}. Let $a_k$ can be distributed between clusters with at most $2$ points in each of them. If all distinct $a_k$ and $a_{k'}$ from different clusters satisfy \eqref{E-condition-T3-1}, then $\B_{M}$ is a basis for $L^2(\cal I_M)$ with 
\begin{equation*}
    \begin{cases} 
    ( M - \alpha)\left(1-\left\lvert \cos{\frac{\pi u}{N}} \right\rvert\right)  \leq A \text{ and } B \leq   (M + \alpha)\left(1+\left\lvert \cos{\frac{\pi u}{N}} \right\rvert \right) \,  & N \text{ is even}\\
   ( M - \alpha)\left(1-\left\lvert \frac{\cos{\left(\frac{\pi}{2N}+\frac{\pi u}{N}\right)}}{\cos{\frac{\pi}{2N}}} \right\rvert\right)  \leq A \text{ and } B \leq   (M + \alpha)\left(1+\left\lvert \frac{\cos{\left(\frac{\pi}{2N}+\frac{\pi u}{N}\right)}}{\cos{\frac{\pi}{2N}}} \right\rvert \right)\, & N \text{ is odd}
    \end{cases},
\end{equation*}
where 
$$
\alpha:= \max_{a_k\neq a_{k'}}\left| \frac{\sin{\left(\pi (a_k - a_{k'}) \frac{M}{N}\right)}}{\sin{\left(\pi (a_k - a_{k'}) \frac{1}{N}\right)}}  \right|.
$$
\end{Cor}
\begin{proof}
   Using the proof of Theorem \ref{T-main-3} we show that any two clusters are sufficiently well-separated to apply Lemma \ref{L-Lemma5.1}. Next, we calculate the smallest and largest singular values of the $2\times M$ matrices corresponding to clusters. Since there are finitely many nodes, there are distinct $a_k$ and $a_{k'}$ for which 
$$
\alpha:= \max_{a_k\neq a_{k'}}\left| \frac{\sin{\left(\pi (a_k - a_{k'}) \frac{M}{N}\right)}}{\sin{\left(\pi (a_k - a_{k'}) \frac{1}{N}\right)}}  \right| = \left| \frac{\sin{\left(\pi (a_k - a_{k'}) \frac{M}{N}\right)}}{\sin{\left(\pi (a_k - a_{k'}) \frac{1}{N}\right)}}  \right|.
$$
Such $a_k$ and $a_{k'}$ produce a matrix that yields the largest and smallest singular values. A direct calculation yields 
$$
\sigma_{min}' =  M - \alpha   \text{  and  } \sigma_{max}' =  M + \alpha.
$$
Therefore, by Lemma \ref{L-Lemma5.1} and Theorem \ref{T-matrix} $\B_{M}$ is a basis for $L^2(\cal I_M)$ with frame constants 
\begin{equation*}
    \begin{cases} 
    ( M - \alpha)\left(1-\left\lvert \cos{\frac{\pi u}{N}} \right\rvert\right)  \leq A \text{ and } B \leq   (M + \alpha)\left(1+\left\lvert \cos{\frac{\pi u}{N}} \right\rvert \right) \,  & N \text{ is even}\\
   ( M - \alpha)\left(1-\left\lvert \frac{\cos{\left(\frac{\pi}{2N}+\frac{\pi u}{N}\right)}}{\cos{\frac{\pi}{2N}}} \right\rvert\right)  \leq A \text{ and } B \leq   (M + \alpha)\left(1+\left\lvert \frac{\cos{\left(\frac{\pi}{2N}+\frac{\pi u}{N}\right)}}{\cos{\frac{\pi}{2N}}} \right\rvert \right)\, & N \text{ is odd}
    \end{cases}.
\end{equation*}
\end{proof}

\subsection{Proof of Theorem \ref{T-main-2}}
We start with the interval $I_N=[0,N)$.  Recall that  the set 
$$
\B = \bigcup_{j=0}^{N-1} \left\{e^{2\pi i \left(n+\frac{j}{N}\right)x}\right\}_{n\in\Z}
$$
is an orthogonal basis for $L^2(I_N)$. From $I_N$ we remove the interval $I_{m}=(m,m+1)$, where $m<N-1$ and $m\in\N$. We find a basis for $L^2(I_N\setminus I_{m})$ in the form 
$$
\B_{\delta} = \bigcup_{j=0}^{N-2} \left\{e^{2\pi i\left(n+\frac{j}{N-1}-j\delta\right)x}\right\}_{n\in\Z},
$$
where $|\delta|$ is sufficiently small. This is equivalent to showing that the matrix
$$
\G_{\delta} = \left\{e^{2\pi i k\left(\frac{j}{N-1}-j\delta\right)}\right\}_{0\leq j,k \leq N-1}
$$
is not singular. 

First, we observe that for small $\delta$ we can form $N-1$ clusters. The first $N-2$ of them will contain only one point from $1$ to $N-2$, and the last cluster will include the points $0$ and $N-1$. Note that the distance between any two points from different clusters is $\frac{1}{N-1}-\delta$ in a wrap-around sens. 

Let $\beta$ be the solution of the equation
 $$
    \frac{\sin{(\pi (N-1) \beta)}}{\sin{(\pi \beta)}}=(N-1)\sin{\frac{1}{N-1}}
    $$     
    on the interval $\left(0,\frac{1}{N-1}\right)$. Using Lemma \ref{L-p_k=msN} and Lemma \ref{L-p_k=1} we get that for $\frac{1}{2(N-1)^2}<\delta< \frac{1}{N-1}-\beta$ and for any $k\neq k'$ such that $|p_k-p_{k'}|<N-1$  
\begin{align*}
    |\langle L_k,L_{k'} \rangle| &= \left|\frac{\sin{(\pi (p_k-p_{k'})(N-1)\delta)}}{\sin{\left(\pi (p_k-p_{k'})\left(\frac{1}{N-1}-\delta\right)\right)}}\right| \leq \left|\frac{\sin{(\pi (N-1)\delta)}}{\sin{\left(\pi \left(\frac{1}{N-1}-\delta\right)\right)}}\right| \\ 
    &< (N-1)\sin{\frac{1}{N-1}}. 
\end{align*}
Next, we evaluate the singular values of the submatrices of $\G$ formed by the clusters. Matrices formed by one-node clusters, only have one singular value equal to $N-1$. The singular values of the $2\times N-1$ matrices formed by two-nodes clusters can be evaluated as in the proof of Theorem \ref{T-main}. We obtain
$$
{\sigma'}_{1,2}^2=\lambda_{1,2}=N-1 \pm \left\lvert \frac{\sin{(\pi (N-1)^2\delta)}}{\sin{(\pi (N-1)\delta)}} \right\rvert.
$$
Thus, by Lemma \ref{L-Lemma5.1}, $\B_{\delta}$ is a basis for $L^2(I_N\setminus I_{m})$ with frame constants
$$
\left(1-\left|\frac{\sin{(\pi (N-1)\delta)}}{\sin{\left(\pi \left(\frac{1}{N-1}-\delta\right)\right)}}\right| \right) \left(N-1 - \frac{\sin{(\pi (N-1)^2\delta)}}{\sin{(\pi (N-1)\delta)}} \right)  \leq A  
$$
and
$$
 B\leq \left(1+\left|\frac{\sin{(\pi (N-1)\delta)}}{\sin{\left(\pi \left(\frac{1}{N-1}-\delta\right)\right)}}\right| \right) \left(N-1 + \frac{\sin{(\pi (N-1)^2\delta)}}{\sin{(\pi (N-1)\delta)}} \right),
$$
whenever $\frac{1}{2(N-1)^2}<\delta< \frac{1}{N-1}-\beta$. Furthermore, using Lemma \ref{L-p_k=msN} and Lemma \ref{L-p_k=1}, the constants can be simplified as follows: 
$$
\left(1-(N-1)\sin{\frac{1}{N-1}} \right) \left(N-1 - \frac{1}{\sin{\frac{\pi}{2(N-1)}}} \right)  \leq A  
$$
and 
$$
 B\leq \left(1+(N-1)\sin{\frac{1}{N-1}} \right) \left(N-1 + \frac{1}{\sin{\frac{\pi}{2(N-1)}}} \right).
$$
\begin{FlushRight}
$\square$
\end{FlushRight}

\section{Remarks and open problems}

We begin this section by providing details on some of the examples that were mentioned in the introduction. The first example demonstrates how even minor perturbations in the domain can affect an exponential basis.  
\begin{Ex}\label{Ex-Theorem-main}
    Let $\cal I_2 = [0,1)\cup [3,4)$ and $\cal I_{2,\frac{1}{N}} = [0,1)\cup \left[3-\frac{1}{N},4-\frac{1}{N}\right)$ for some $N\in\N$. Then 
    $$
\B = \bigcup_{j=0}^{1} \left\{e^{2\pi i\left(n+\frac{j}{2}\right)x}\right\}_{n\in\Z}
$$
is an orthogonal basis for $L^2(\cal I_2)$, but it is not a basis for $L^2(\cal I_{2,\epsilon})$.
\end{Ex}

   By using Theorem \ref{T-matrix} the matrix associated to   $\B$ and $I_2$ is
    $$
    \G = \begin{pmatrix}
        1 & 1  \\
        1& -1 \\
    \end{pmatrix}.
    $$
 This matrix is orthogonal, so $\B$ is an orthogonal basis for $L^2(\cal I_2)$.

    We apply the linear transformation $\cal L(x)=N x$ to $I_{2, 1/N}$ and get $\cal I_{2N,1} = [0,N)\cup [3N-1,4N-1)$. Using Theorem \ref{T-matrix} we form the matrix $\G$ for $\cal I_{2N,1}$ and to  
        $$
\B = \bigcup_{j=0}^{2N-1} \left\{e^{2\pi i\left(n+\frac{j}{2N}\right)x}\right\}_{n\in\Z}.
$$
$\G=\{ e^{2\pi i   \frac{p_k(j-1)}{2N}}\}_{1\leq j,k\leq L}$ is a Vandremonde matrix with endpoints $p_k=\{0,...,N-1,3N-1,...,4N-2\}$. The endpoints $N-1$ and $3N-1$ of $I_{2N, 1}$  correspond to the same node $\frac{N-1}{2N}$ because  
$$
\left|\frac{3N-1}{2N} - \frac{N-1}{2N}  \right|_{\T}=0.
$$
Thus, $\G$ is singular and so $\B$ is not a basis for $L^2(\cal I_{2,\epsilon})$.
\medskip

The following example shows that removing portions of the domain cannot be compensated solely by adjusting the density of exponents based on the measure of the new domain. So, additional perturbation of the phase may sometimes be required.

\begin{Ex}\label{EX-Theorem1.2}
   The set $\B_3=\bigcup_{j=1}^3\left\{ e^{2 \pi i \left(n+\frac{j-1}{3}\right)x} \right\}$ is the standard orthogonal basis for $L^2([0,3))$. However, $\B_2=\bigcup_{j=1}^2\left\{ e^{2 \pi i \left(n+\frac{j-1}{2}\right)x} \right\}$, is not  a basis for $L^2([0,1)\cup [2,3))$.
\end{Ex}

By using Theorem \ref{T-matrix} the matrix associated to   $\B_2$ and $[0,1)\cup [2,3)$ is
\begin{equation*}
    \G_2 = \begin{pmatrix}
        1 & 1  \\
        1& 1 \\
    \end{pmatrix}.
\end{equation*}
It is easy to see that  $\G_2$ is singular. Therefore, $\B_2$ is not  a basis for $L^2([0,1)\cup [2,3))$.
\medskip

The next proposition follows from Theorem \ref{T-matrix}. However, as was mentioned earlier the most difficult part of this research is to estimate the Riesz constants of the basis. 
\begin{Prop}\label{P-basis}
    Let $M,N\in\N$ such that $M<N$. Then 
    $$
\B_{M} = \bigcup_{j=1}^{M} \left\{e^{2\pi i\left(n+\frac{j-1}{N}\right)x}\right\}_{n\in\Z},
$$
is a basis for $L^2(\cal I_M)$, where $\cal I_M$ is as in \eqref{E-I}.
\end{Prop}
\begin{proof}
    Using Theorem \ref{T-matrix} we form a matrix 
$$
\G=\left\{ e^{2\pi i \left(\frac{j-1}{N}\right)a_k}\right\}_{1\leq j,k\leq M}
$$
that corresponds to the set $\B_{M}$ on $\cal I_M$. $\G$ is a Vandermonde matrix with distinct nodes, so it is non-singular, and $\B_{M}$ is a basis for $L^2(\cal I_M)$.
    
\end{proof}
\medskip

The next proposition is well-known, but for the reader's convenience, we will prove it.
\begin{Prop}\label{P-basisMod}
   Let $I_s$ be as in \eqref{E-I}. If $a_j$ (mod $s$) are distinct, then 
$$
\B=\bigcup_{j=0}^{s-1}\left\{e^{2 \pi i \left(n+\frac{j}{s}\right)x}\right\}_{n\in\Z}
$$
is an orthogonal exponential basis for $L^2(I_s)$ with $A=B=s$.
\end{Prop}
\begin{proof}
Let  $a_j$ (mod $s$) are distinct.  By using Theorem \ref{T-matrix} the matrix associated to  $\B$ and $I_s$ is
    $$
\G=\left\{ e^{2\pi i \left(\frac{a_k(j-1)}{s}\right)}\right\}_{1\leq j,k\leq s}=\left\{ e^{2\pi i \left(\frac{(j-1)}{s}\right)}\right\}_{1\leq j,k\leq s}.
$$
$\G$ is an orthogonal Vandermonde matrix with singular values equal to $s$, so $\B$ is an orthogonal exponential basis for $L^2(I_s)$ with the frame constants $A=B=s$. 
\end{proof}

\medskip

Theorem \ref{T-main} demonstrates that if we slightly adjust the intervals in the domain, we can still obtain a basis by modifying a standard exponential basis on these intervals with a special $\delta$, and the Riesz constants can be estimated accordingly. Moreover, when all $\epsilon_j=0$, Theorem \ref{T-main} provides a basis for the original domain too.

It is important to note that Theorem \ref{T-main} applies only with standard exponential bases. However, a similar approach can be used for any domain with an exponential basis that can be associated with a Vandermonde matrix, as shown in Theorem \ref{T-matrix}. This method could potentially be applied to unions of intervals of length 1, but further research is needed, which we plan to pursue in the future. 
\medskip

In Theorem \ref{T-main-3}, we considered a standard exponential basis on union of intervals inside an interval of length $N$. The condition $2<M\leq\frac{N}{2}$ illustrates that the union of intervals is just a ``small" part of the interval $[0,N)$. The question, is how to find exponential bases on the ``large" portion, i.e. is the complement of a ``small" part of the interval $[0,N)$?

To answer this question we use Theorem 9 in \cite{PRW2024} that we state below. 
\begin{Thm}\label{T-Complement}
Let $\Delta>0$ and $S\subset [0,\Delta).$ Suppose that for some $\Lambda\subset \frac{1}{\Delta}\Z,$ $\E(\Lambda)$ is a Riesz basis for $L^2(S)$ with Riesz's constants $A$ and $B$. Then $\E(\frac{1}{\Delta}\Z\setminus\Lambda)$ is a Riesz basis for $L^2([0,\Delta)\setminus S)$ with $A'=\Delta-B$ and  $B'=\Delta-B$. 
\end{Thm}
The proof of such result can be found in \cite[Proposition 2.1]{MM2009} and \cite[Corollary 5.6]{BCMS2019}. Theorem \ref{T-main-2} addresses the case of removing a single interval of length 1. However, the method presented in this article can be extended to handle the removal of multiple intervals, which we will explore in future work. Additionally, we will explore the application of the method in $\R^d$.

 \bibliographystyle{plain}
\bibliography{Main}

\end{document}